\documentclass[12pt,leqno,twoside]{amsart}
\usepackage{amssymb}
\usepackage{latexsym}
\usepackage{graphicx}
\usepackage{epsfig}
\usepackage{srcltx}
\usepackage{subcaption} 

\newtheorem{theorem}{Theorem}[section]
\newtheorem{corollary}[theorem]{Corollary}
\newtheorem{lemma}[theorem]{Lemma}
\newtheorem{proposition}[theorem]{Proposition}

\theoremstyle{remark}
\newtheorem{remark}[theorem]{\sc Remark}
\theoremstyle{remark}

\theoremstyle{definition}

\theoremstyle{remark}

\theoremstyle{remark}

\theoremstyle{remark}

\numberwithin{equation}{section}  

\renewcommand{\Box}{\square}    

\newcommand{\cal}{\mathcal}

\newcommand{\corank}{\mathrm{corank\hspace{2pt}}}

\newcommand{\h}{{\mathrm{ht}}}

\renewcommand{\int}{{\mathrm{int}}}

\newcommand{\Sing}{{\mathrm{Sing\hspace{2pt}}}}

\newcommand{\im}{{\mathrm{Im\hspace{1pt}}}}
\renewcommand{\ker}{\mathop{{\mathrm{ker}}}\nolimits}
\newcommand{\coker}{\mathop{{\mathrm{coker}}}\nolimits}

\newcommand{\e}{\varepsilon}
\newcommand{\m}{\setminus}

\newcommand{\fin}{\hspace*{\fill}$\Box$\vspace*{2mm}}

\newcommand{\tF}{F^{\pitchfork}}
\newcommand{\tE}{E^{\pitchfork}}
\newcommand{\tmu}{\mu^{\pitchfork}}

\newcommand{\cA}{{\cal A}}
\newcommand{\cB}{{\cal B}}

\newcommand{\cC}{{\cal C}}

\newcommand{\cU}{{\cal U}}

\newcommand{\cX}{{\cal X}}
\newcommand{\cN}{{\cal N}}
\newcommand{\cY}{{\cal Y}}
\newcommand{\cZ}{{\cal Z}}

\newcommand{\bC}{{\mathbb C}}

\newcommand{\bP}{{\mathbb P}}

\newcommand{\bZ}{{\mathbb Z}}

\newcommand{\bQ}{{\mathbb Q}}

\begin{document}

\title[Milnor Fibre via Deformation]
 {Milnor Fibre Homology via Deformation}

\author{\sc Dirk Siersma}  
\address{Institute of Mathematics, Utrecht University, PO
Box 80010, \ 3508 TA Utrecht
 The Netherlands.}

\email{D.Siersma@uu.nl}

\author{\sc Mihai Tib\u ar}

\address{Univ. Lille, CNRS, UMR 8524 - Laboratoire Paul Painlev\'e, F-59000 Lille, France}

\email{tibar@math.univ-lille1.fr}

\thanks{}



\date{\today}

\dedicatory{Dedicated to Gert-Martin Greuel on the occasion of his 70th birthday}



\begin{abstract}
In case of one-dimensional singular locus, we use deformations in order to get refined information about the Betti numbers of the Milnor fibre.

\end{abstract}
 
\maketitle

\setcounter{section}{0}

\section{Introduction and results}\label{s:intro}

We study the topology of Milnor fibres $F$ of function germs on $\bC^{n+1}$ with a 1-dimensional singular set.
Well known is that $F$ is a $(n-2)$ connected $n$-dimensional CW-complex. What can be said about $H_{n-1}(F)$ and $H_{n}(F)$? In this paper we use deformations in order to get  information about these groups. It turns out that  the constraints on $F$ yield only small  numbers $b_{n-1}(F)$, for which we give upper bounds which are in general sharper than the known ones from \cite{Si3}.  
The upper Betti number $b_{n}(F)$ can  be determined from an Euler characteristic formula. We pay special attention to  classes of singularities where  $H_{n-1}(F)=0$, where the homology is concentrated in the middle dimension.

The admissible deformations of the function have a singular locus $\Sigma$ consisting of  a finite set $R$ of isolated points and finitely many curve branches. Each branch $\Sigma_i$ of $\Sigma$ has a generic transversal type (of transversal Milnor fibre $\tF_i$ and Milnor number denoted by $\tmu_i$) and also contains a finite set $Q_i$ of points with non-generic transversal type, which we call \emph{special points}. In the neighbourhood of each such special point $q$ with Milnor fibre denoted by $\cA_q$, there are two monodromies which act on $\tF_i$: the \emph{Milnor monodromy} of the local Milnor fibration of $\tF_i$, and the \emph{vertical monodromy} of the local system defined on the germ  of $\Sigma_{i}\m \{q\}$ at $q$.
 
 In our topological study we work with homology over $\bZ$ (and therefore we systematically omit  $\bZ$ from the notation of the homology groups). 
We provide a detailed expression for $H_{n-1}(F)$ through a topological model of $F$ from which we derive the following results.
\noindent 
\begin{itemize}

\item[a.] If for every component $\Sigma_i$  there exist one vertical monodromy $A_s$, which has no eigenvalues $1$, then $b_{n-1}(F)=0$. More generally: $b_{n-1}(F)$ is bounded by the sum (taken over the components) of the minimum (over that component) of $\dim \ker (A_s-I)$ (Theorem \ref{t:bound}).
\item[b.]
Assume that for each irreducible  component $\Sigma_i$  there is a special singularity at $q$ such that    $H_{n-1}(\cA_q)=0$. Then $H_{n-1}(F) =0$.\\
More generally: 
Let $Q' := \{q_1,\ldots, q_m\}\subset Q$ be a subset of special points such that each branch $\Sigma_{i}$ contains at least one of its points.  Then (Theorem \ref{p:conc}b):
\[ b_{n-1}(F) \le \dim H_{n-1}(\cA_{q_1}) + \cdots + \dim H_{n-1}(\cA_{q_m}).\] 
\end{itemize}
Note that in both cases already some (small) subset of the special points may have a strong effect and that we may choose {\it the best bound.}

\smallskip

In \cite{ST-vh} we have studied the vanishing homology of projective hypersurfaces with a 1-dimensional singular set. The same type of methods work in the local case. We keep the notations close to those in \cite{ST-vh} and refer to it for the proof of certain results.
In the proof of the main theorems we use the Mayer-Vietoris theorem to study local and (semi) global contributions separately. We construct a CW-complex model of two bundles of transversal Milnor fibres (in \S \ref{ss:cw}  and \S \ref{ss:cw2}) and  their inclusion map (\S \ref{ss:proofmain}). Moreover we use the full strength of the results on local 1-dimensional singularities \cite{Si1}, \cite{Si2}, \cite{Si3}, \cite{Si-Cam}, cf also \cite{NS}, \cite{Ra}, \cite{Ti-nonisol}, \cite{Yo}.

We discuss known results such as De Jong's \cite{dJ} and compute several examples in \S \ref{remarks}. 

\medskip
\noindent
{\it Acknowledgment.} Most of the research of this paper took place during a Research in Pairs of the authors at the Mathematisches Forschungsinstitut Oberwolfach in November 2015. The authors thank the institute for the support and excellent atmosphere.

\section{Local theory of 1-dimensional singular locus} \label{ss:localtheory}

We work with local data of function germs with 1-dimensional singular locus and we will apply results from the well-known theory which we extract from  \cite{Si3},  \cite{Si-Cam}, and \cite{ST-vh}.

 Let $f : (\bC^{n+1},0) \to (\bC,0)$ be a  holomorphic function germ with singular locus $\Sigma$ of dimension 1 and let  $\Sigma = \Sigma_1 \cup \ldots \cup \Sigma_m $ be its decomposition into irreducible curve components. Let $E := B_{\e}\cap f^{-1}(D_{\delta})$ be the Milnor neighbourhood and  $F$ be the local Milnor fibre of $f$, for small enough $\e$ and $\delta$. 
 The homology $\tilde{H}_*(F)$ is
concentrated in dimensions $n-1$ and $n$. The non-trivial groups are $H_n (F) = \bZ^{\mu_n}$, 
which is free,  and
$H_{n-1} (F)$ which can have torsion. 

There is a well-defined local system on $\Sigma_i \m \{ 0\}$ having as fibre the homology of the transversal Milnor fibre $\tilde{H}_{n-1} (\tF_i)$, where $\tF_i$ is the Milnor fibre of the restriction of $f$ to a transversal hyperplane section at some $x \in \Sigma_i \m \{ 0\}$. This restriction has an isolated singularity
whose equisingularity class is independent of the point $x$ and of the transversal section, in particular  $\tilde{H}_{*} (\tF_i)$
is concentrated in dimension $n-1$. It is on this group that acts the \emph{local system monodromy} (also called \emph{vertical monodromy}):
\[ A_i: \tilde{H}_{n-1} (\tF_i) \rightarrow \tilde{H}_{n-1} (\tF_i). \]

After \cite{Si3}, one considers a tubular neighbourhood $\cN := \sqcup_{i=}^m \cN_i$ of the link of $\Sigma$ and decomposes the boundary $\partial F :=  F \cap \partial B_{\e}$ of the Milnor fibre as
$\partial F = \partial_1 F\cup \partial_2 F$,
where $\partial_2 F := \partial F \cap \cN$. Then $\partial_2 F = \displaystyle{\mathop{\sqcup}^{m}_{i=1}} \partial_2 F_i$, where $\partial_2 F_i := \partial_2 F \cap \cN_i$.

  Each boundary component $\partial_2 F_i$ is fibred over the link of $\Sigma_i$ with fibre $\tF_i$.
Let then $\tE_i$ denote the transversal Milnor neighbourhood containing the transversal fibre $\tF_i$ and let $\partial_2 E_i$  denote the total space of its fibration above the link of $\Sigma_i$. Therefore $\tE_i$ is contractible and $\partial_2 E_i$ retracts to the link of $\Sigma_i$. The pair $(\partial_2 E_i,\partial_2 F_i)$  is related to $A_i - I$ via  the following exact relative Wang sequence \cite{ST-vh} ( $n \geq 2$):
\begin{equation} \label{l:localrelative}  
  0 \rightarrow H_{n+1} (\partial_2 E_i,\partial_2 F_i) \rightarrow H_{n} (\tE_i,\tF_i)
\stackrel{A_i -I}{\rightarrow} H_{n} (\tE_i,\tF_i) \rightarrow H_{n-1}
(\partial_2 E_i, \partial_2 F_i) \rightarrow 0 .
\end{equation}

\section{Deformation and vanishing homology}\label{s:defvh} 

Consider now a 1-parameter family $f_s:  (\bC^{n+1},0) \to (\bC,0)$ where  $f_0 = \hat{f} : (\bC^{n+1},0) \to (\bC,0)$ is a given germ with singular locus $\hat{\Sigma}$ of dimension 1,  with  Milnor data $(\hat{E},\hat{F})$ and $\hat{\Sigma }= \hat{\Sigma_1} \cup \ldots \cup \hat{\Sigma_m }$ and all the other objects defined like in \S \ref{ss:localtheory}.  
We use the notation with ``hat'' since we reserve the notation without ``hat'' for the deformation $f_{s}$.   

We fix a ball $B:= B_{\e}\subset \bC^{n+1}$ centered at $0$ and a disk $\Delta :=\Delta_{\delta}\subset \bC$  at $0$ such that, for small enough radii $\e$ and $\delta$ the restriction to the punctured disc $\hat{f}_{|} : B \cap f^{-1}(\Delta^{*}) \to \Delta^{*}$ is the Milnor fibration of $\hat{f}$.

We say that the deformation $f_{s}$ is {\it admissible} if it has good behavior at the boundary, i.e.,  if for small enough $s$ the family ${f_{s}}_{|} : \partial B \cap f^{-1}(\Delta) \to \Delta$ is stratified topologically trivial.\footnote{
Such a situation occurs e.g.in the case of an ``equi-transversal deformation'' considered in \cite{MaSi}.}

We choose a value of $s$ which satisfies the above conditions and write from now on
 $f :=f_s$.
It then follows that the pair $(E,F) := (B\cap f^{-1}(\Delta), f^{-1}(b))$, where $b\in \partial \Delta$, is topologically equivalent to the  Milnor data $(\hat{E},\hat{F})$ of $\hat{f}$.  Note that for $f$ we consider the semi-local singular fibration inside $B$ and not just its Milnor fibration at the origin.

\medskip

Let $\Sigma \subset B$ be the 1-dimensional singular part of the singular set $\Sing(f)\subset B$.
Note that $\hat{\Sigma} = \bigcup_{i \in \hat{I}} \hat{\Sigma}$ and  $\Sigma = \bigcup_{i \in I} \Sigma_i$ can have a different number of irreducible components.  It follows that the circle boundaries $\partial B \cap \hat \Sigma$  of $ \hat\Sigma$ identify to the circle boundaries $\partial B \cap \Sigma$  of $\Sigma$ and that the corresponding  vertical monodromies  are the same.

\subsection{Notations}\label{ss:notations} We use notations similar to \cite{ST-vh} (cf also figure \ref{fig:defo}).\\
A point $q$ on $\Sigma$ is called {\it special} if the transversal Milnor fibration is not a local product in a neighbourhood of that point. 
 
 \noindent
  $Q_i:=$  the set of special points on $\Sigma_i$; $Q:= \cup_{i\in I} Q_i$, \\
 $R:=$  the set of isolated singular points; $R = R_0 \cup R_1$, where $R_0$ are the critical points on $f^{-1}(0)$ and  $R_1$ the critical points outside $f^{-1}(0)$,

 \noindent
$ B_q, B_r$ = small enough disjoint Milnor balls within $E$  at the points $q \in Q, r \in R$ resp. \\
 $B_Q := \sqcup_q B_q$ and $B_R := \sqcup_r B_r$, and similar notation for $B_{R_0}$ and $B_{R_1}$,

 \noindent
$\Sigma^*_i := \Sigma_i \m B_Q$; $\Sigma^* = \cup_{i\in I} \Sigma_i^*$, \\
 $\cU_i :=$ small enough  tubular neighbourhood of $\Sigma^*_i$; $\cU = \cup_{i} \cU_i$,\\
  $\pi_\Sigma : \cU \to \Sigma^*$ is the projection of the tubular neighbourhood.\\
$T= \{f(r) | r \in R \} \cup \{f(\Sigma)\}$ is the set of critical values of $f$ and we assume without loss of generality that $f(\Sigma)= 0$.

\medskip
Let$\{\Delta_{t}\}_{t\in T}$ be a system of non-intersecting small discs $\Delta_t$ around each $ t \in T$.
For any $t \in T $, choose $t'\in \partial\Delta_t$. If $t = f(r)$ then we denote by  $t'(r)$  the point $t' \in \Delta_{f(r)}$. For $t=0$ we use the notations $t_{0}$ and $t'_{0}$ respectively.

Let $E_r = B_r \cap f^{-1}(\Delta_{ f(r)})$ and  $F_r = B_r \cap f^{-1}(t'(r))$ be the Milnor data of the isolated singularity of $f$ at $r\in R$.
We  use next the additivity of vanishing homology with respect to the different critical values  and the connected
 components of $\Sing f$.
\begin{figure}[htbp]
    \centering
        \includegraphics[width=11 cm]{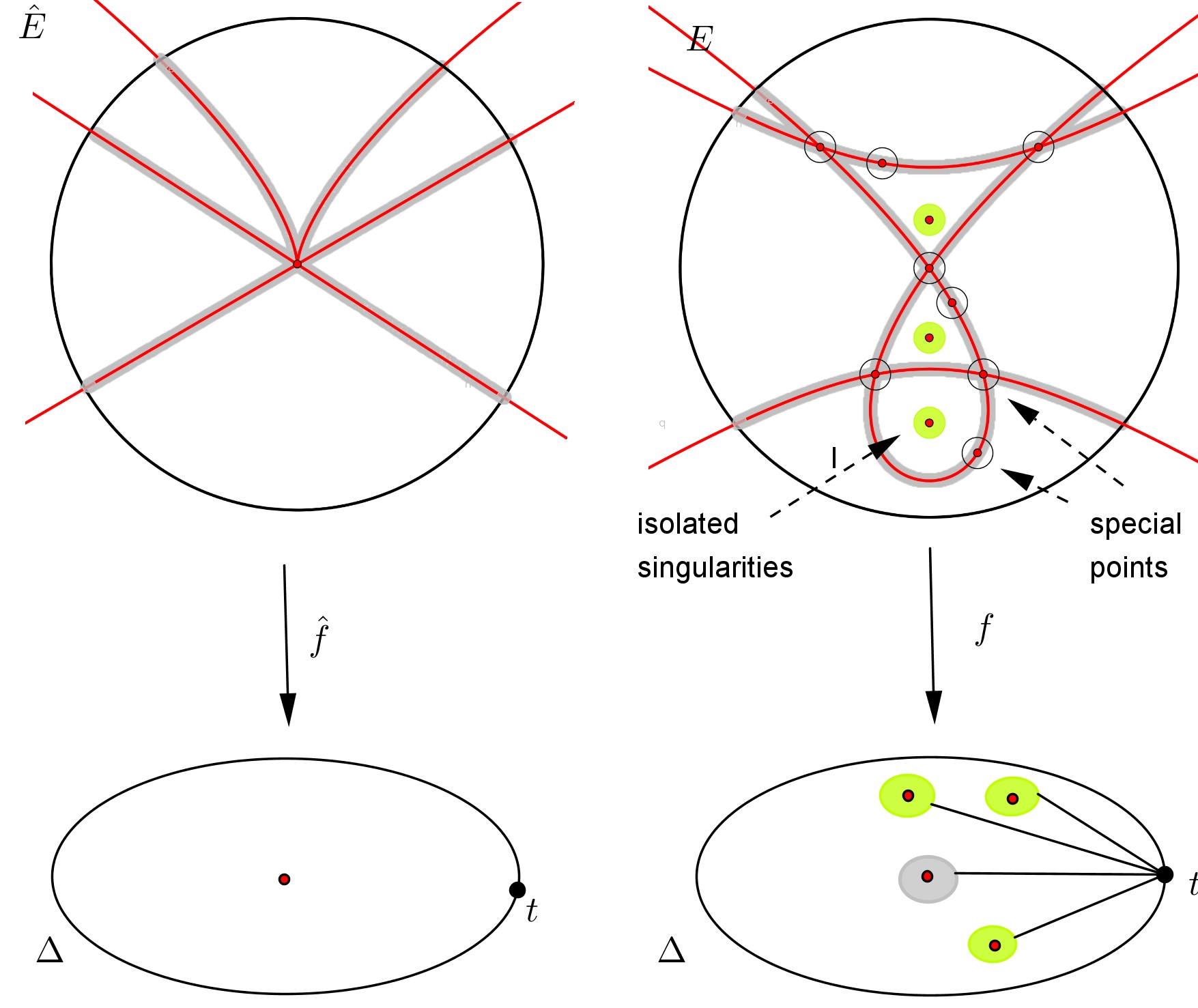}
    \caption{Admissible deformation}
    \label{fig:defo}
\end{figure}
By homotopy retraction and by excision we have:
\begin{equation}\label{eq:firstexision}
  H_*(E, F) \simeq \oplus_{t \in T}H_*((f^{-1}(\Delta_t), f^{-1}(t') ) =
\end{equation}	
\begin{equation} \label{eq:directsumdecomp}
	= \oplus_{r\in R_0} H_*(E_r,F_r) \oplus H_*(E_0,F_0) \oplus \oplus_{r\in R_1} H_*(E_r,F_r),
\end{equation} 
where $(E_0,F_0) = (f^{-1}(\Delta_{0}) \cap (\cU  \cup B_Q), f^{-1}(t'_{0})  \cap (\cU \cup  B_Q)$
\noindent
We introduce the following shorter notations:
 \[ 
(\cX_q,\cA_q) := (f^{-1}(\Delta_{0}) \cap B_q,f^{-1}({t'_{0}})  \cap B_q) \;  \]
\[ \; \cX = \sqcup_Q \cX_q  \;  \; , \; \;  \cA= \sqcup_Q \cA_q   \]
\[ \cY = \cU \cap f^{-1}(\Delta_0 ) \; , \;  
 \cB :=  f^{-1}({t'_{0}}) \cap \cY \;  \;
  \]
\[   \cZ := \cX\cap \cY   \; \; , \; \cC := \cA\cap \cB  \]
\smallskip
 In these new notations we have: 
\begin{equation}\label{eq:directsumdecomp2}
  H_*(E, F) \simeq H_*(\cX\cup \cY, \cA\cup \cB) \oplus \oplus_{r\in R} H_*(E_r, F_r).
\end{equation}
Note that each direct summand $H_*(E_r,F_r)$ is concentrated in dimension $n+1$ since it identifies to the Milnor lattice $\bZ^{\mu_{r}}$ of the isolated singularities germs of $f-f(r)$ at $r$, where $\mu_r$ denotes its Milnor number.
\smallskip
 We deal from now on with the  term $H_*(\cX\cup \cY, \cA\cup \cB)$ in the direct sum of \eqref{eq:directsumdecomp2}. \\
We consider the relative Mayer-Vietoris long exact sequence:
\begin{equation}\label{eq:mv}
 \cdots \to H_{*}(\cZ,\cC) \to H_{*}(\cX,\cA) \oplus H_{*}(\cY,\cB) \to  H_{*}(\cX\cup \cY, \cA\cup \cB) \stackrel{\partial_s}{\to} \cdots
\end{equation}
 of the pair $(\cX\cup \cY, \cA\cup \cB)$ and we compute each term of it in the following.  The description follows closely \cite{ST-vh} where we have treated deformations of projective hypersurfaces.

\subsection{The homology of $(\cX,\cA)$}\label{ss:term}
One has the direct sum decomposition $H_{*}(\cX,\cA) \simeq \oplus_{q\in Q} H_*(\cX_q,\cA_q)$  
since $\cX$ is a disjoint union. 
The pairs $(\cX_q,\cA_q)$ are local Milnor data of the hypersurface germs $(f^{-1}(t_0), q)$ with 1-dimensional singular locus and therefore the relative homology $H_*(\cX_q,\cA_q)$ is concentrated in dimensions $n$ and $n+1$.
 
\subsection{The homology of $(\cZ,\cC)$}\label{ss:term2}
The pair $(\cZ,\cC)$ is a disjoint union of pairs localized at points  $q \in Q$. 
 For such points we have one contribution for each {\it locally irreducible branch of the germ $(\Sigma,q)$}. Let  $S_q$ be the index set of all these branches at $q\in Q$. 
By abuse of notation we will also write $s \in S_q$ for the corresponding small loops around $q$ in $\Sigma_i$. 
 For some $q\in \Sigma_{i_1} \cap \Sigma_{i_2}$, the set of indices $S_q$ runs over all the local irreducible components of the curve germ $(\Sigma, q)$. 
Nevertheless, when we are counting the local irreducible branches at some point $q\in Q_i$ on a specified component $\Sigma_i$ then the set $S_q$ will tacitly mean only those local branches of $\Sigma_i$ at $q$. 
We get the following decomposition:
\begin{equation}\label{eq:sumdecomp}
  H_{*}(\cZ,\cC) \simeq  \oplus_{q\in Q} 
\oplus_{s\in S_q} H_{*}(\cZ_s, \cC_s).
\end{equation}

More precisely, one such local pair $(\cZ_s, \cC_s)$ is the bundle over the corresponding component of the link of the curve germ $\Sigma$ at $q$ having as fibre the local transversal Milnor data $(\tE_s, \tF_s)$, with transversal Milnor numbers denoted by $\tmu_s$. These data depend only on the branch $\Sigma_i$ containing $s$,  and therefore if $s \subset \Sigma_i$ we sometimes write $(\tE_i, \tF_i)$ and $\tmu_i$ .
In the notations of \S \ref{ss:localtheory}, we have:
 $\partial_2 \cA_q =\mathop{\sqcup}_{s \in S_q} \cC_s $.

The relative homology groups in the above direct sum decomposition \eqref{eq:sumdecomp} depend on the \emph{local system  monodromy} $A_s$ via the following Wang sequence which is a relative version of \eqref{l:localrelative} and has been proved in \cite[Lemma 3.1]{ST-vh}:
\begin{equation}\label{eq:wang2}
 0 \rightarrow H_{n+1} (\cZ_s,\cC_s)) \rightarrow H_{n} (\tE_s,\tF_s)
\stackrel{A_s -I}{\rightarrow} H_{n} (\tE_s,\tF_s) \rightarrow H_{n}
(\cZ_s,\cC_s) \rightarrow 0.
\end{equation}
 
From this we get: 

\begin{lemma}\label{l:concentration}
At $q \in Q$,  for each $s\in S_q$ one has:
\begin{eqnarray*}
  H_{k}(\cZ_s,\cC_s ) = 0 &  k\not= n, n+1,\\ 
H_{n+1}(\cZ_s,\cC_s) \cong {\ker} \; (A_s - I), &   H_{n}(\cZ_s,\cC_s ) \cong {\coker} \; (A_s - I).
\end{eqnarray*}
\fin
\end{lemma}

We conclude that $H_*(\cZ,\cC)$ is concentrated in dimensions $n$ and $n+1$ only.
\smallskip

\begin{figure}[htbp]
    \centering
        \includegraphics[width=13 cm]{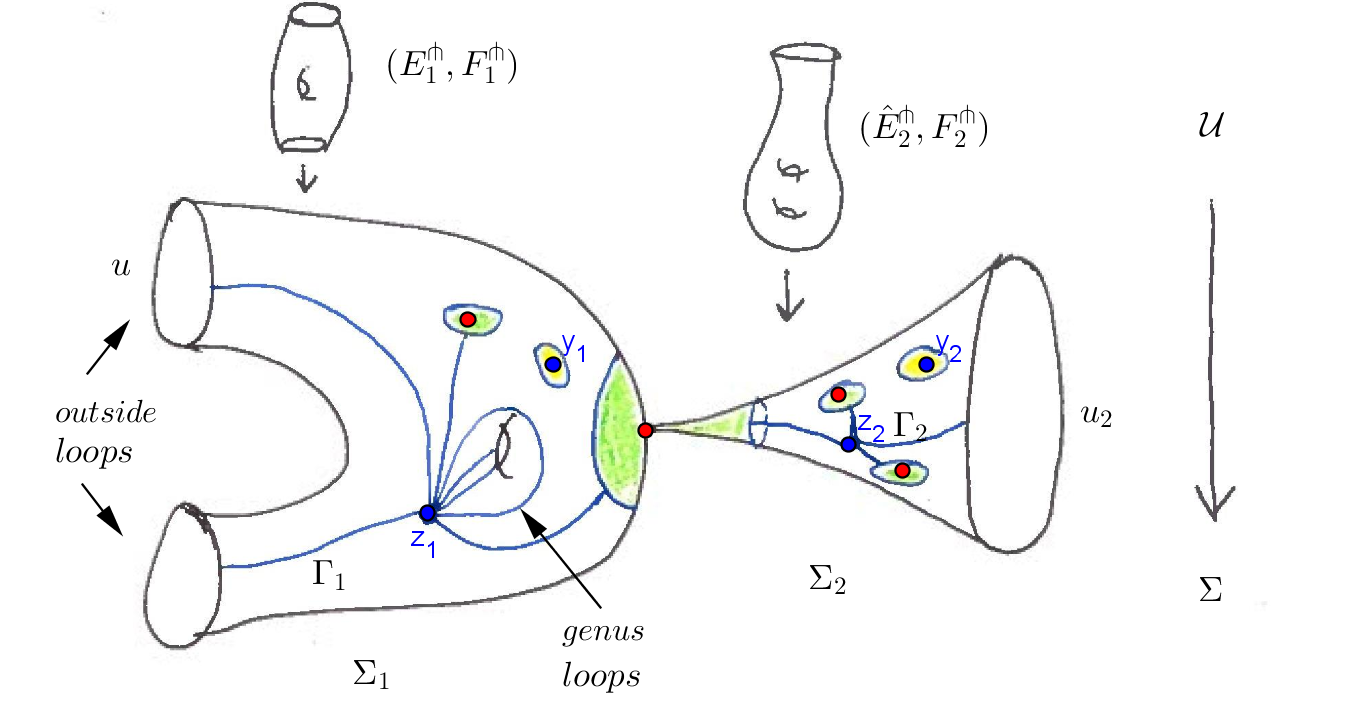}
    \caption{Critical set and the cell models for  $(\cZ,\cC)$ and $(\cY,\cB)$.}
    \label{fig:critset}
\end{figure}

\subsection{The CW-complex structure of $(\cZ,\cC)$} \label{ss:cw}
The pair $(\cZ_s,\cC_s)$ has the following structure of a relative CW-complex, up to homotopy.
Each bundle over some circle link can be obtained from a trivial bundle over an interval by identifying 
the fibres above the end points via the geometric  monodromy $A_s$.
In order to obtain $\cZ_s$ from $\cC_s$ one can start by first attaching $n$-cells  $c_{1}, \ldots, c_{\tmu_s}$
 to the fibre $\tF_s$ in order to kill the $\tmu_s$ generators of $H_{n-1}(\tF_s)$ at the identified ends, 
  and next by attaching $(n+1)$-cells $e_{1}, \ldots, e_{\tmu_s}$ to the preceding  $n$-skeleton. The attaching of some $(n+1)$-cell goes as follows: consider some $n$-cell $a$ of the $n$-skeleton and take the cylinder $I\times a$ as an $(n+1)$-cell. Fix an orientation of the    circle link, attach the base $\{ 0\}\times a$ over $a$, then follow the circle bundle in the fixed orientation by the monodromy $A_s$ and attach the end $\{ 1\}\times a$ over $A_s(a)$.
At the level of the cell complex, the boundary map of this attaching identifies 
to $A_s -I : \bZ^{\tmu_s} \to \bZ^{ \tmu_s}$.

\subsection{The CW-complex structure of $(\cY,\cB)$} \label{ss:cw2} 

The curve $\Sigma$ has as boundary components the intersection $\partial B \cap \Sigma$ with the Milnor ball.
They are all topological circles. We denote them with $u \in U_i$, $U := \cup_{i} U_i$ and call them {\it outside} loops. Note that over any such loop $u \in U_i$ we have a local system monodromy 
 $A_u:  \bZ^{\tmu_i} \rightarrow \bZ^{\tmu_i}$. In fact this monodromy did not change in the admissible deformation from $\hat{f}$ to $f$.

 For technical reasons we introduce one more puncture $y_i$  on $\Sigma_i$ and next redefine $\Sigma^*_i := \Sigma\m (Q \cup \{y_i\})$ 
Moreover we use notations $(\cX_y,\cA_y)$ and $(\cZ_y,\cC_y)$.
We choose the following sets of loops\footnote{We identify the loops with their index sets.} in $\Sigma_i$:
\begin{itemize}
\item[ $G_i$] the $2g_i$ loops (called \emph{genus loops} in the following) which are
 generators of  $\pi_1$ of the normalization $\tilde\Sigma_i$ of $\Sigma_i$,   where $g_i$ denotes the genus of this normalization (which is a Riemann surface with boundary),
\item[ $S_i$] the loops $s$ around the special points $q \in Q_i$,
\item[ $U_i$] the outside loops,
\end{itemize}
and define $W_i = G_i \sqcup S_i \sqcup U_i$ and  $W= \sqcup W_i$.
By enlarging ``the hole'' defined by the  puncture $y_i$, we retract $\Sigma^*_i$ to some configuration of  loops connected by non-intersecting paths to some point $z_{i}$, denoted by $\Gamma_i$ (see Figure \ref{fig:critset}). The number of loops is  $\# W_i = 2g_{i} + \tau_i +\gamma_i$, where $\tau_i := \# U_i$ and $\gamma_i := \sum_{q\in Q_i}\# S_q$. 
Note that $\tau_i >0$ since there must be at least one outside loop.

 Each  pair $(\cY_i,\cB_i)$ is then homotopy equivalent (by retraction) 
 to the pair  $(\pi_\Sigma^{-1}(\Gamma_i),\cB\cap \pi_\Sigma^{-1}(\Gamma_i))$.  We endow the latter with the structure of a relative CW-complex as we did with $(\cZ,\cC)$  at \S \ref{ss:cw}, namely for each  
 loop the similar CW-complex structure as we have defined above for some pair $(\cZ_s,\cC_s)$.
The difference is that the pairs $(\cZ_s,\cC_s)$ are disjoint whereas in $\Sigma^*_i$ 
the loops meet at a single point $z_i$. We thus  take as reference the transversal
 fibre $\tF_i  = \cB \cap\pi_\Sigma^{-1}(z_i)$ above this point, namely we attach the $n$-cells (thimbles) only once to this single fibre in order to kill the $\tmu_i$ generators of $H_{n-1}(\tF_i)$. 
The $(n+1)$-cells of $(\cY_i,\cB_i)$ correspond to the fibre bundles over the loops in the bouquet model of $\Sigma^*_i$. Over each loop,  one attaches a number of $\tmu_i$ $(n+1)$-cells to the fixed $n$-skeleton 
described before, more precisely one $(n+1)$-cell over one $n$-cell generator of the $n$-skeleton.
We extend for $w \in W$ the notation $(\cZ_g,\cC_g)$ to genus loops and $(\cZ_u,\cC_u)$ to outside loops, although they are not contained in $(\cZ,\cC)$ but in $(\cY, \cB)$.

Here the attaching map of the $(n+1)$-cells corresponding to the bundle over a genus loop, or over an outer loop,
  can be identified with $A_g -I : \bZ^{\tmu_i} \to \bZ^{\tmu_i}$, or with $A_u -I : \bZ^{\tmu_i} \to \bZ^{\tmu_i}$, respectively. We have seen that the monodromy $A_{u}$ over some outer loop indexed by $u\in U_{i}$ is necessarily one of the vertical monodromies of the original function $\hat f$.

From this CW-complex structure we get the following precise description in terms of the  monodromies of the transversal local system, the proof of which is similar to that of \cite[Lemma 4.4]{ST-vh}:
\begin{lemma}\label{p:concentr}\ 
\begin{enumerate}
\item
 $H_{k}(\cY,\cB) = \oplus_{i\in I}H_{k}(\cY_i,\cB_i)$ and this is $=0$   for $k \ne n, n+1.$  
\item 
$H_{n}(\cY_i,\cB_i) \simeq  \bZ^{\tmu_i} /  \langle \im (A_w-I) \mid  w\in W_i \rangle$,
\item
$\chi(\cY_i,\cB_i) = (-1)^{n-1}(2g_i +\tau_i + \gamma_i -1) \tmu_i$.
\end{enumerate}
\fin
\end{lemma}

If we apply $\chi$ to \eqref{eq:directsumdecomp2} and \eqref{eq:mv} and  take into account that $\chi(\cZ, \cC)=0$, we get: \\ 
 $\chi(E,F) = \chi(\cX,\cA) + \chi(\cY,\cB) + \sum_r \chi(E_r,F_r)$. From this we derive 
the Euler characteristic\footnote{already computed  in \cite{MaSi}} of the Milnor fibre  $F$:
\begin{proposition}\label{p:euler}
\[ \chi(F) = 1+ \sum_{q\in Q} (\chi(\cA_q)-1) + (-1)^{n}\sum_{i\in I} (2g_i +\tau_i + \gamma_i -2) \tmu_i + (-1)^{n}\sum_{r\in R} \mu_{r}.\]
\fin
\end{proposition}

\begin{proposition}\label{p:7-term}
The relative Mayer-Vietoris sequence \eqref{eq:mv} is trivial except of the following  6-terms sequence:
\begin{equation}\label{eq:MV7}\begin{array}{l}
 \ \ \ \ 0 \to H_{n+1}(\cZ,\cC) \to H_{n+1}(\cX,\cA) 
\oplus H_{n+1}(\cY,\cB) \to  H_{n+1}(\cX \cup \cY , \cA \cup \cB) \to \\
 \ \ \ \ \ \to  H_{n}(\cZ,\cC)\stackrel{j}{\to} H_{n}(\cX,\cA) \oplus H_{n}(\cY,\cB)
 \to H_{n}(\cX \cup \cY , \cA \cup \cB) \to  0.
 \end{array}
\end{equation}\fin
\end{proposition}
\begin{proof}
Lemma \ref{l:concentration}, \S \ref{ss:term} and Lemma \ref{p:concentr}
show that the  terms $H_{*}(\cX ,\cA)$, $H_{*}(\cY ,\cB)$ and $H_{*}(\cZ ,\cC)$ of the Mayer-Vietoris sequence \eqref{eq:mv} are concentrated only in dimensions $n$ and $n+1$.
Following \eqref{eq:directsumdecomp2} and since $\tilde H_{*}(F)$ is concentrated  in levels $n-1$ and $n$, we obtain that $H_{n+2}(\cX \cup \cY , \cA \cup \cB) =0$.
\end{proof}
The first 3 terms of \eqref{eq:MV7} are free. By the decomposition \eqref{eq:directsumdecomp2}, in order to find the homology of $F$ we thus need to compute $H_{k}(\cX \cup \cY , \cA \cup \cB)$ for $k=n, n+1$, since the others are zero. 
In the remainder of this paper we find information only about $H_{n}(\cX \cup \cY , \cA \cup \cB)$. The knowledge of  its dimension is then enough for determining $H_{n}(F)$, by only using  the Euler characteristic formula (Prop. \ref{p:euler}).

\section{The homology group $H_{n-1}(F)$}\label{ss:proofmain}

We concentrate on the term $H_{n}(\cX \cup \cY , \cA \cup \cB) \simeq \tilde H_{n-1}(F)$.
We  need the relative version of the ``variation-ladder'', an  exact sequence found in \cite[Theorem 5.2,  p. 456-457]{Si3}. This sequence has an important overlap with our relative Mayer-Vietoris sequence \eqref{eq:MV7}.

\begin{proposition} \label{p:varladderRel}\cite[Proposition 5.2]{ST-vh} 
For any point $q \in Q$, the sequence
\[ \begin{array}{l}   
0 \rightarrow H_{n+1}(\cA_q,\partial_2 \cA_q) \rightarrow \mathop{\oplus} _{s \in S_q} H_{n+1}(\cZ_s,\cC_s)
\rightarrow H_{n+1} (\cX_q,\cA_q) \rightarrow  \ \ \ \ \ \\
 \ \ \ \ \  \rightarrow H_{n}(\cA_q,\partial_2 \cA_q) \rightarrow \  \mathop{\oplus}_{s \in S_q} H_{n}(\cZ_s,\cC_s) \ 
\rightarrow  \ H_{n} (\cX_q,\cA_q)  \ 
\rightarrow  \ 0
   \end{array} \]
is exact for $n\ge 2$.   
\fin
\end{proposition}

\subsection{The image of $j$} \label{ss:surj}

We focus on the map $  j = j_1 \oplus j_2$ which occurs in the 6-term exact sequence \eqref{eq:MV7}, more precisely on the following exact sequence:
\begin{equation} \label{eq:j}
  H_{n}(\cZ,\cC) \stackrel{j}{\to} H_{n}(\cX,\cA) \oplus H_{n}(\cY,\cB)\to H_n(F)\to 0.
\end{equation}
since we have the isomorphism:
\begin{equation} 
H_{n-1}(F) \simeq \coker j.
\end{equation} 
Therefore full information about $j$ makes is possible to compute $H_{n-1}(F)$. But although $j$ is of geometric nature, this information is not always easy to obtain. Below we treat its two components in separately. After that we will make two statements (Theorems \ref{t:bound} and \ref{p:conc}) of a more general type.

\subsubsection{The first component $j_1: H_{n}(\cZ,\cC) \to H_{n}(\cX,\cA)$} \label{sss:first} \ \\
Note that, as shown above, we have the following direct sum decompositions of the source and the target:
\[ \begin{array}{l}
  H_{n}(\cZ,\cC) = \oplus_{q\in Q} \oplus_{s\in S_q}  H_{n}(\cZ_s,\cC_s) \oplus \oplus_{i\in I}H_n(\cZ_{y_{i}},\cC_{y_{i}})                 , \\
  H_{n}(\cX,\cA) = \oplus_{q\in Q} H_n(\cX_q,\cA_q)  \oplus \oplus_{i\in I} H_n(\cX_{y_{i}},\cA_{y_{i}})  .
\end{array}
\]

 As shown in Proposition \ref{p:varladderRel}, at the special points $q\in Q$ we have surjections:
$\mathop{\oplus}_{s \in S_q}  H_{n}(\cZ_s,\cC_s)  \rightarrow H_{n} (\cX_q,\cA_q)$  and moreover
$H_n(\cZ_y,\cC_y) \rightarrow H_n(\cX_y,\cA_y)$ is an isomorphism.
We conclude to the surjectivity of the morphism $j_1$ and to the cancellation of the contribution of the points $y_{i}$ for $\coker j$.

\subsubsection{The second component $j_2: H_{n}(\cZ,\cC) \to  H_{n}(\cY,\cB)$}  \ \\
Both sides are described with a relative CW-complex as explained in \S \ref{ss:cw2}. At the level of $n$-cells there are $\tmu_s$  $n$-cell generators of $H_{n}(\cZ_s,\cC_s)$ for each $s \in S_q$ and any $q\in Q$. Each of these generators is mapped bijectively to the single cluster of $n$-cell generators attached to the reference fibre $\tF_{i}$ (which is the fibre above the common point $z_i$ of the loops).
The restriction ${j_2}_| :  H_{n}(\cZ_s,\cC_s) 
	\to  H_{n}(\cY_i,\cB_i)$ is a projection  for any loop $s$ in $\Sigma_i$ and $q\in Q_{i}$,  or if instead of $s$ we have $y_{i}$,
since we add extra relations to $\bZ^{\tmu}/\langle A_s-I \rangle$ in order to get $\bZ^{\tmu_i} /  \langle \im (A_w-I) \mid  w\in W_i \rangle  =H_{n}(\cY_i,\cB_i)$.
 We summarize the above surjections as follows:

\begin{lemma} \label{l:surjective} \rm (``Strong surjectivity'') \it
\begin{enumerate}
\item Both $j_1$ and $j_2$ are surjective. 
\item The restriction ${j_{2}}_{|} :   H_{n}(\cZ_s,\cC_s)  \rightarrow H_{n} (\cY_{i},\cA_i)$    is surjective for any $s \in S_q$ such that $q\in Q\cap \Sigma_{i}$.
\item  The restriction  $j_{1}| \oplus_{s\in S_{q}}H_n(\cZ_s,\cC_s) \rightarrow H_n(\cX_{q},\cA_{q})$ is surjective, for any $q\in Q$.
\end{enumerate}
\fin
\end{lemma}

\begin{corollary} \label{c:surjective} 
\begin{enumerate}
\item If the restriction ${j_{2}}| \ker j_{1}$ is surjective, then $j$ is surjective.
\item If for each $i\in I$ there exists $q_{i}\in Q\cap \Sigma_{i}$ and some $s \in S_{q_{i}}$  such that 
$H_{n}(\cZ_s,\cC_s) \subset \ker j_{1}$ then $j$ is surjective. 
\end{enumerate}
\fin

\end{corollary}
\begin{proof}
(a).  More generally, let   $j_{1}: M \to M_{1}$  and $j_{2}: M \to M_{2}$
be morphisms of  $\bZ$-modules  such that $j_{1}$ is surjective and consider the direct sum of them $j:= j_{1}\oplus j_{2}$.  
We assume that the restriction ${j_{2}}| \ker j_{1}$ is surjective onto $M_{2}$ and want to prove that $j$ is surjective. 

Let then $(a,b) \in M_{1}\oplus M_{2}$. There exists $x\in M$ such that $j_{1}(x)=a$, by the surjectivity of $j_{1}$. Let $b' := j_{2}(x)$. By our surjectivity assumption there exists $y\in \ker j_{1}$ such that $j_{2}(y) = b-b'$.
 Then  $j(x+y) = a+b$,  which proves the surjectivity of j.

\noindent
(b). follows immediately  from Lemma \ref{l:surjective}(b) and from the above (a).
\end{proof}

\subsection{Effect of local system monodromies on $H_n(F)$}
Recall that $w\in W_{i}$ stands for some loop $s, g , u$ in $\Sigma_i^{*}$.

\begin{theorem}\label{t:bound}\
\begin{enumerate}
\item
If there is $w\in W_{i}$ such that $\det (A_w - I) \not=0$ then $\dim H_n(\cY_i,\cB_i) = 0$. \\
If such  $w\in W_{i}$ exists for any $i \in I$, then $b_{n-1}(F)= 0$.
\item If there is $w\in W_{i}$ such that $\det (A_w - I) = \pm 1$ then $H_n(\cY_i,\cB_i) = 0$. \\
If such  $w\in W_{i}$ exists for any $i \in I$, then $H_{n-1}(F) = 0$.
\item The following upper bound holds:
 \[b_{n-1}(F)\le \sum_{i\in I} \min_{w\in W_i}\dim \coker (A_w - I) \le \sum_{i \in I} \tmu_i.\]
\end{enumerate}
\end{theorem}
\begin{proof}
By Lemma \ref{p:concentr}(b). we have
$H_{n}(\cY_i,\cB_i) \simeq  \bZ^{\tmu_i} /  \langle \im (A_w-I) \mid  w\in W_i \rangle$, thus
the first parts of (a)  and (b)  follow. For the second part of (a), we have that $\dim H_n(\cY,\cB) = 0$, hence $\corank j = \corank j_{1} =0$. For the second part of (b), we have that $H_n(\cY,\cB) = 0$ and the surjectivity of the map $j$ of \eqref{eq:j} is equivalent to the fact that $j_{1}$ is surjective. 

\noindent
To prove (c), we consider  homology groups  with coefficients in $\bQ$. Since $j_{1}$ is surjective, the  image of $j$ contains all the
generators of  $H_n(\cX,\cA;\bQ)$. Hence $\dim \coker j \le \dim H_n(\cY,\cB)$.
\end{proof}
\begin{remark}
Notice the {\it effect of the strongest bound} in the above theorem. On each $\Sigma_i$ one could take an optimal loop, e.g. one with $\det (A_w - I) = \pm 1$.
Since in  the deformed case there may be less branches $\Sigma_{i}$, and  more special points and hence more vertical monodromies, these bounds may become much stronger than those in \cite{Si3}.
\end{remark}

\bigskip

\subsection{Effect of the local fibres $\cA_q$}

\begin{theorem} \label{p:conc} Let $n\ge 2$.
\begin{enumerate}
\item
Assume that for each irreducible 1-dimensional component $\Sigma_i$ of $\Sigma$ there is a special singularity $q\in Q_{i}$ such that the $(n-1)$th homology group of its Milnor fibre is trivial, i.e.   $H_{n-1}(\cA_q)=0$. Then $H_{n-1}(F) =0$.

\noindent
If in the above assumption we replace  $H_{n-1}(\cA_q)=0$ by $b_{n-1}(\cA_q) = 0$, then we get  $b_{n-1}(F) = 0$.
\item
Let $Q' := \{q_1,\ldots, q_m\}\subset Q$ be some (minimal) subset of special points such that each branch $\Sigma_{i}$ contains at least one of its points.  Then:
\[ b_{n-1}(F) \le \dim H_n(\cX_{q_1},\cA_{q_1}) + \cdots + \dim H_n(\cX_{q_m},\cA_{q_m}).\]
\end{enumerate}
\end{theorem}
\begin{proof}(a).  We use  \eqref{eq:j} in order to estimate the dimension of the image of $j= j_1 \oplus j_2$. 
If  there is a $q\in Q$ such that $H_n(\cX_q,\cA_q)= 0$ then $\ker j_{1}$ contains  $\oplus_{s \in S_q} H_n(\cZ_s,\cC_s)$.
Since $ Q'$ meets all components $\Sigma_i$, statement (a) follows from Corollary \ref{c:surjective}(b). The second claim of (a) follows by considering homology over $\bQ$.

\noindent
(b). We work again with homology over $\bQ$. 
We consider the projection on  a direct summand $\pi : H_n(\cX,\cA)\to \oplus_{q\not\in Q'}H_n(\cX_{q},\cA_{q})$ and the composed map $J_{1}:= \pi \circ j_{1}$. Then the restriction $j_{2}| \ker J_{1}$ is surjective, which by Corollary \ref{c:surjective}(a), means that $J_{1}\circ j_{2}$ is surjective. Then the result follows from the obvious inequality $\dim (\im J_{1}\circ j_{2}) \le \dim \im j$ by counting dimensions.
\end{proof}

\begin{remark}
Also here we have the {\it effect of the strongest bound}. This works at best if one chooses an optimal  or minimal $Q'$ (see e.g. Figure \ref{f:Qpoints}).
In the irreducible case,   $H_{n-1}(\cA_q) = 0$ for at least one $q\in Q$ already implies the triviality $H_{n-1}(F)=0$.
\begin{figure}[htbp]
    \centering
        \includegraphics[width=6cm]{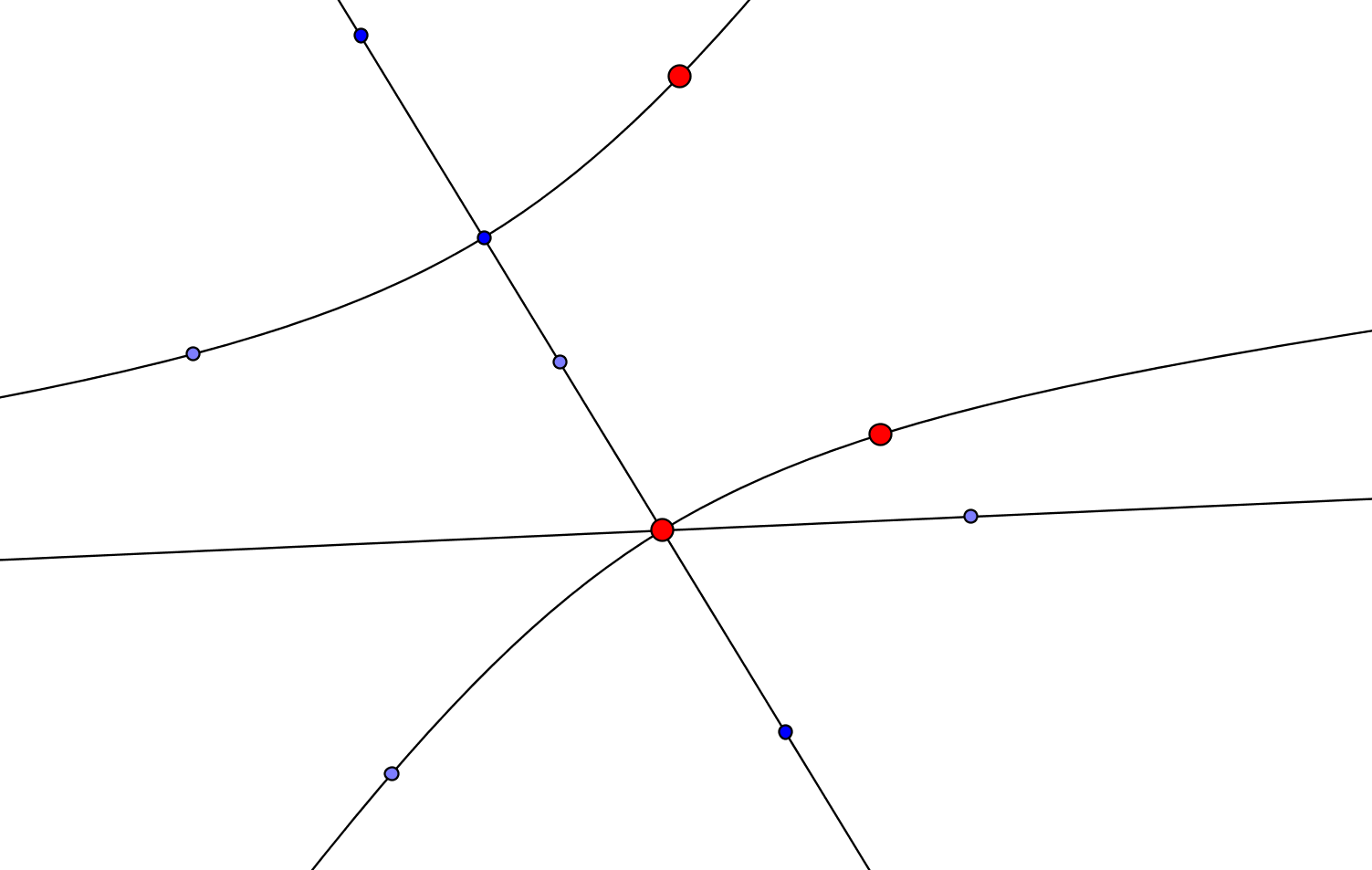}
    \caption{A choice of Q-points}
    \label{f:Qpoints}
\end{figure}
\end{remark}

\begin{corollary} (Bouquet Theorem) 
If $n \geq 3$ and
\begin{itemize}
\item[(a)] If for any $i \in I$ there is $w\in W_{i}$ such that $\det (A_w - I) = \pm 1$, or 
\item[(b)] If for every $\Sigma_i$ there is a special singularity $q\in Q_{i}$ such that  $H_{n-1}(\cA_q)=0$
\end{itemize}
then  
$$ F  \stackrel{\h}{\simeq}   S^n \vee \cdots \vee S^n.$$
\end{corollary}
\begin{proof}
From Theorems (\ref{t:bound}b) or (\ref{p:conc}a) follows $H_{n-1}(F)=0$. Since $F$ is a simply connected $n$-dimensional CW-complex the statement follows from Milnor's argument (\cite{Mi}, theorem 6.5) and Whitehead's theorem.
\end{proof}

\section{Examples }\label{remarks}

\subsection{Singularities with transversal type $A_1$ }
The case when $\Sigma$ is a smooth line was considered in \cite{Si1} and later generalized to $\Sigma$ a 1-dimensional complete intersection (icis) \cite{Si-icis}. It uses an admissible deformation with only $D_{\infty}$-points. The main statement is:
\begin{itemize}
\item[(a)] $ F \stackrel{\h}{\simeq} S^{n-1} $  if   $ \# D_{\infty} = 0 $,
\item[(b)]  $ F \stackrel{\h}{\simeq} S^n \vee \cdots \vee S^n $ else.
\end{itemize}
Since $D_{\infty}$-points have $H_{n-1}(\cA_q)=0$, our Theorem \ref{p:conc} provides a proof of this statement on the level of homology.
\smallskip
If $\Sigma$ is not an icis, more complicated situations occur. For details about the following example, cf \cite{Si-icis}.
\begin{itemize}
\item[(i)] $f = xyz$, called $T_{\infty,\infty,\infty}$ : $\Sigma$ is the union of 3 coordinate axis. $F \cong S^1 \times S^1$, so $b_1(F) = 2$, $b_2(F)=1$ and all $A_u = I$.
\item[(ii)] $f= x^2y^2 + y^2z^2 + x^2z^2$  has $F \cong S^2 \vee \cdots \vee S^2$.
The admissible deformation $f_s = f + s x y z$ has the same $\Sigma$ as $f= xyz$, but now  with 3 $D_{\infty}$-points on each component of $\Sigma$ and one $T_{\infty,\infty,\infty}$-point in the origin. Our Theorem \ref{p:conc} therefore states $H_{1}(F)=0$.
A real picture of $f_s=0$ contains the Steiner surface, for $s\not=0$ small enough (Figure \ref{fig:steiner}). That $H_{2}(F) = \bZ^{15}$ follows from $\chi(F) =16$ computed via Proposition \ref{p:euler}.

\begin{figure}[h]
  \begin{subfigure}[b]{0.25\textwidth}
    \includegraphics[width=\textwidth]{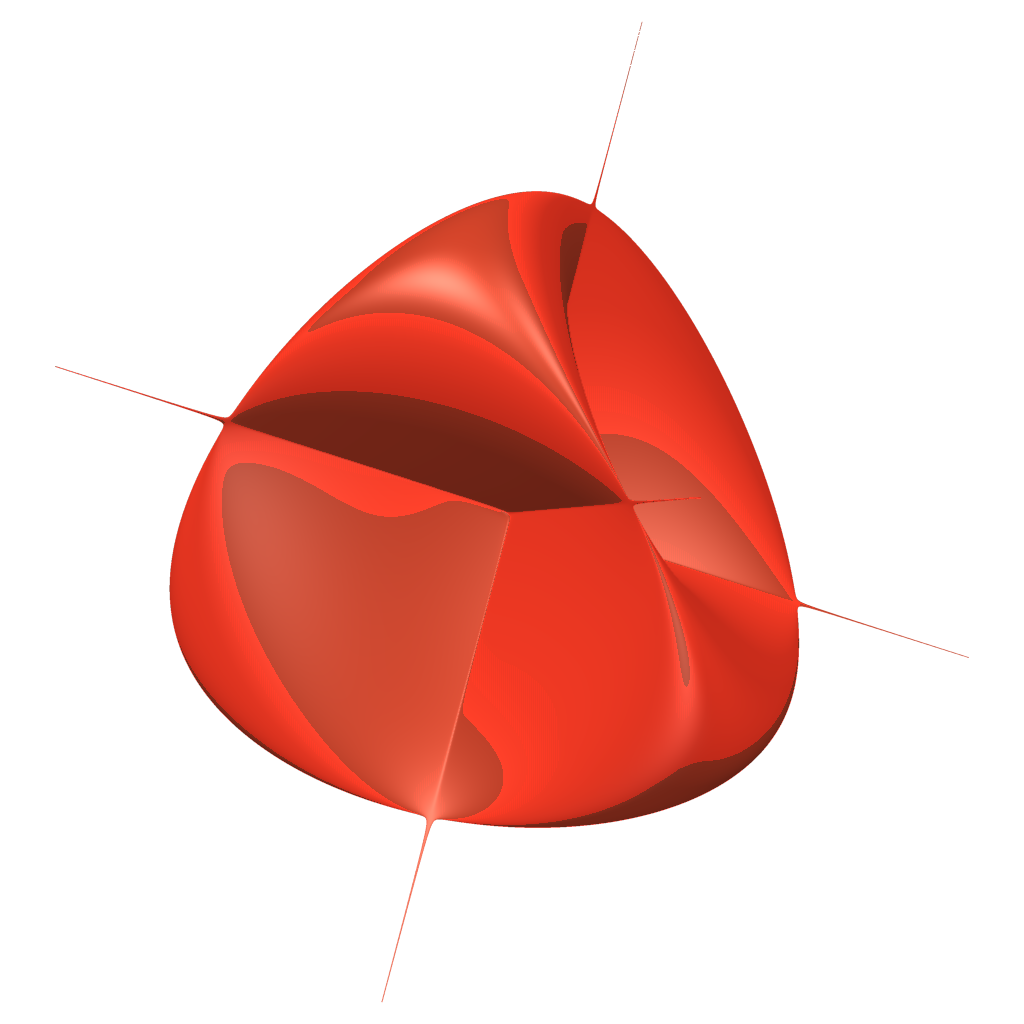}
   \caption{Steiner Surface}
     \label{fig:steiner}
  \end{subfigure}
	 \begin{subfigure}[b]{0.25\textwidth}
    \includegraphics[width=\textwidth]{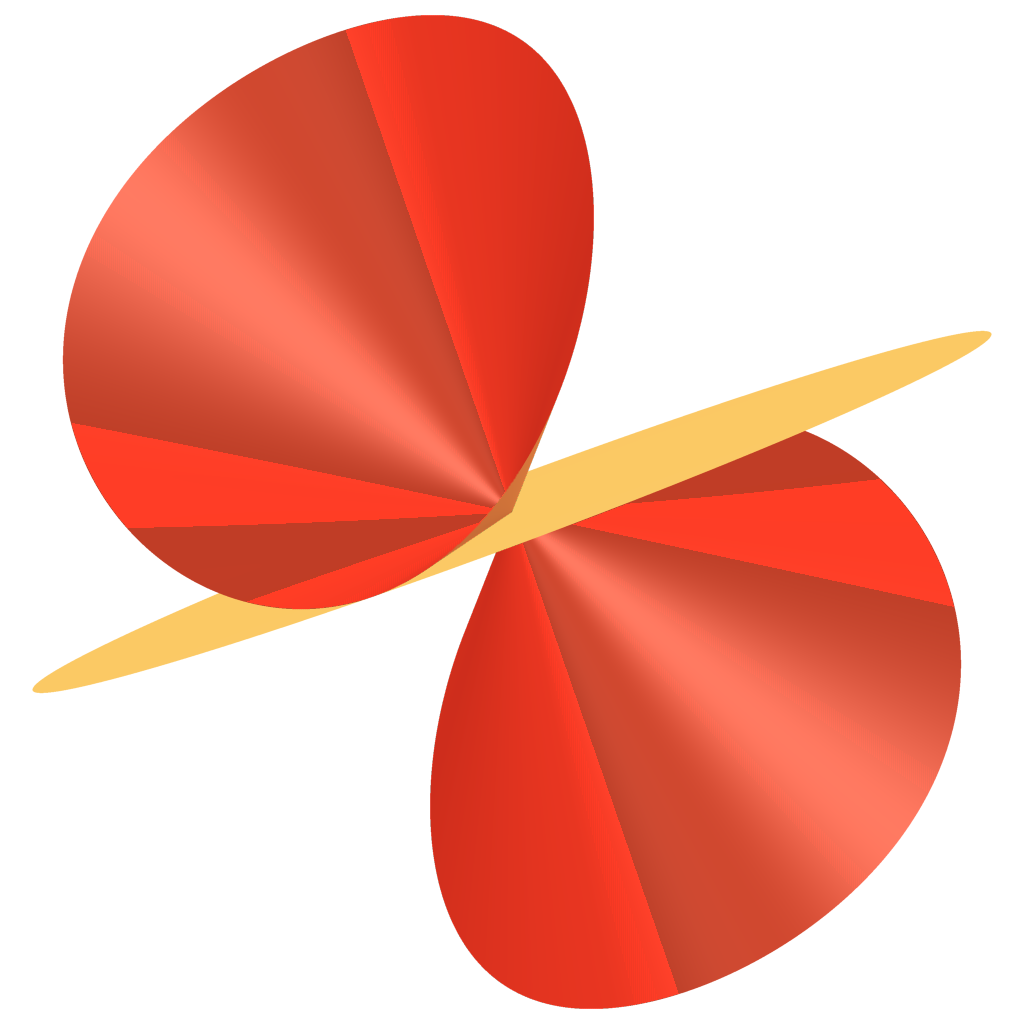}
    \caption{Singularity $F_1 A_3$}
    \label{fig:dJnewF1A3}
  \end{subfigure}
  \begin{subfigure}[b]{0.25\textwidth}
    \includegraphics[width=\textwidth]{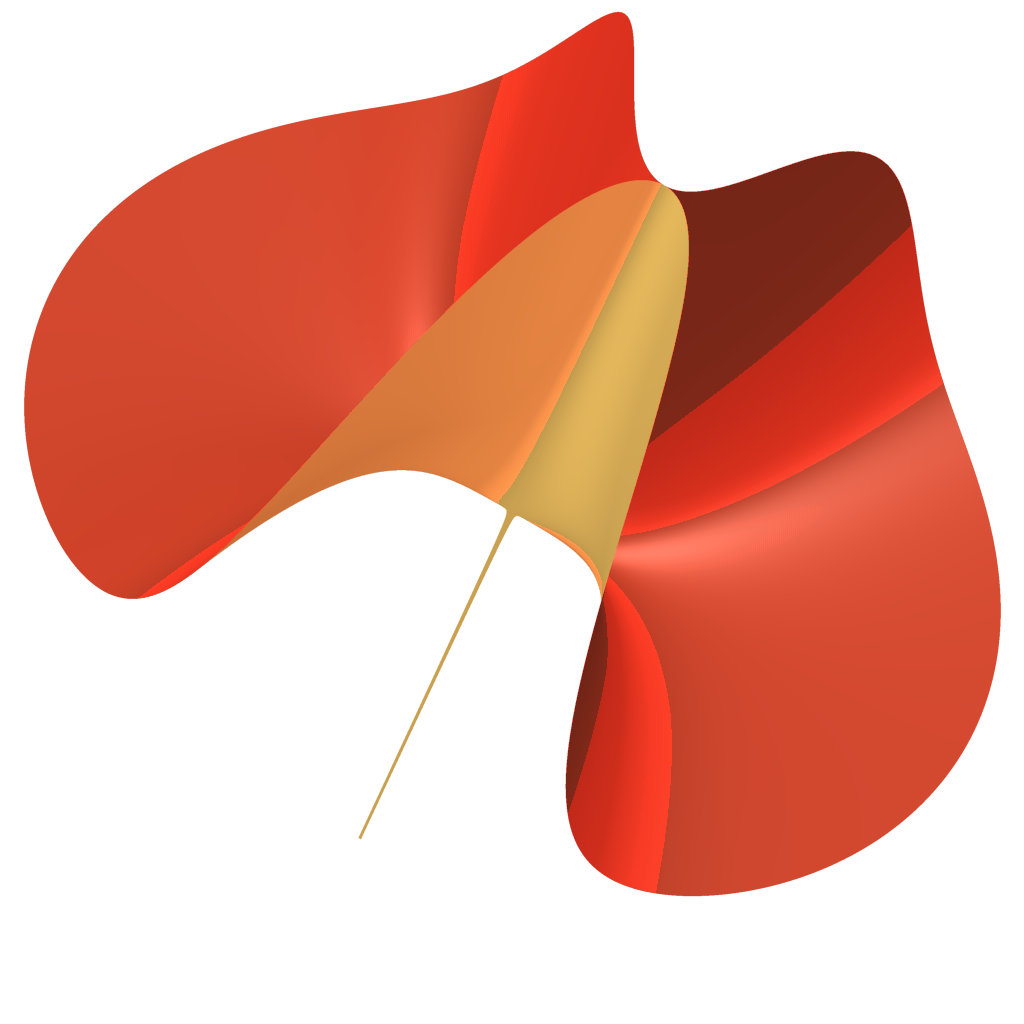}
    \caption{Singularity $F_2A_3$}
    \label{fig:DJF2A3}	
  \end{subfigure}
\caption{Several Singularites (\tiny{produced with Surfer software})}
\end{figure}
\end{itemize}

\subsection{Transversal type $A_2,A_3,D_4,E_6,E_7,E_8$, De Jong List}
In \cite{dJ} there is a detailed description of singularities with singular set a smooth line
and transversal type $A_2,A_3,D_4,E_6,E_7,E_8$. 
His list illustrates and confirms our statements at the level of homology.

We will treat below in more detail the case $f: \bC^3 \to \bC$ with transversal type $A_3$. (By adding squares, this also illustrates  $f: \bC^{n+1} \to \bC$.) Any singularity of this type can be deformed into
\begin{itemize}
\item[$F_1 A_3$]: $ f = x z^2 + y^2 z  \; ; \;  F \stackrel{\h}{\simeq} S^1$  (figure \ref{fig:dJnewF1A3})
\item [$F_2A_3$]:   $ f = x y^4 + z^2  \; ; \;  F \stackrel{\h}{\simeq} S^2$  (figure \ref{fig:DJF2A3})
\end{itemize}
 De Jong's observation is that for any line singularity of transversal type $A_3$we have:
\begin{itemize}
\item[(a)] $ F \stackrel{\h}{\simeq} S^{n-1} \vee S^n \cdots \vee S^n $  if   $\# F_2A_3 = 0 $,
\item[(b)]   $ F \stackrel{\h}{\simeq} S^n \vee \cdots \vee S^n $ else.
\end{itemize}
In homology, (b) follows directly from our concentration result \ref{p:conc}. The homology version of (a) takes more efforts.
We demonstrate  this in the following example only.
First we mention that for $F_1A_3$ the vertical monodromy $A$ is equal to the Milnor monodromy $h$. This follows from the fact that $f = x z^2 + y^2 z$ is homogeneous of degree $d=3$ and Steenbrink's remark \cite{St}  that $Ah^{d}=I$ and that $h^4 = I$.
The matrix of $h$ is:
$$ 
\begin{pmatrix}
1 &  1 & \ 1\\
-1 &  0 & 0\\
0  &  -1 & 0
\end{pmatrix}
$$
It follows: $\ker (h-I) = \bZ$ ; Im $(h-I) = \bZ^2$ and $\coker (h-I) = \bZ.$\\

Next consider as example the deformation $f:= f_s =(x^k-s) z^2 + y z^2 + y^2 z$ for some fixed small enough $s\not=0$, which has transversal type $A_3$. This deformation
 has  $\# F_1A_3 = k $ and $\# F_2A_3 = 0$ and moreover one isolated critical point of type $A_k$.
We compare now the fundamental sequence for $j$ in case $F_1A_3$  and $f$ respectively\footnote{We distinguish the Milnor fibres by a subscript.}:
\begin{equation}
  j = j_1 \oplus j_2: \bZ \rightarrow \bZ \oplus \bZ \rightarrow H_{n-1}(F_{F_1A_3}) = \bZ \rightarrow 0
\end{equation}

\begin{equation}
  j = j_1 \oplus j_2: \bZ^k \rightarrow \bZ^k \oplus \bZ \rightarrow H_{n-1}(F_f)= \bZ \rightarrow 0
\end{equation}

The map $j_2$ for $f$ is as follows:

$\oplus_s H_n(\cZ_s,\cC_s) 
 =  \bZ^k = \oplus_s \bZ^3 / \langle h - I \rangle  \to \bZ^3 / \langle h - I, A_u -I \rangle = H_n(\cY,\cB)$.
It is the sum of components which are isomorphism on each factor $\bZ$. Note  that for the outside loop $u$ we have $A_u - I= (h - 1)( h^{k-1} + \cdots + h + I)$ since $A_{u} = A_{s_{1}}\circ \cdots \circ A_{s_{k}} = h^{k}$  (all $A_{s}$ are equal to $h$).\\
We conclude $H_{1}(F_f) = \bZ$. Next $H_{2}(F_f) = \bZ^{3k-1}$ follows from $\chi(F_f) =3k-1$ computed via Proposition \ref{p:euler}.

We illustrate this example with Figures \ref{fig:dJex2} and \ref{fig:dJex2d}.

\begin{figure}[h]
  \begin{subfigure}[b]{0.3\textwidth}
    \includegraphics[width=\textwidth]{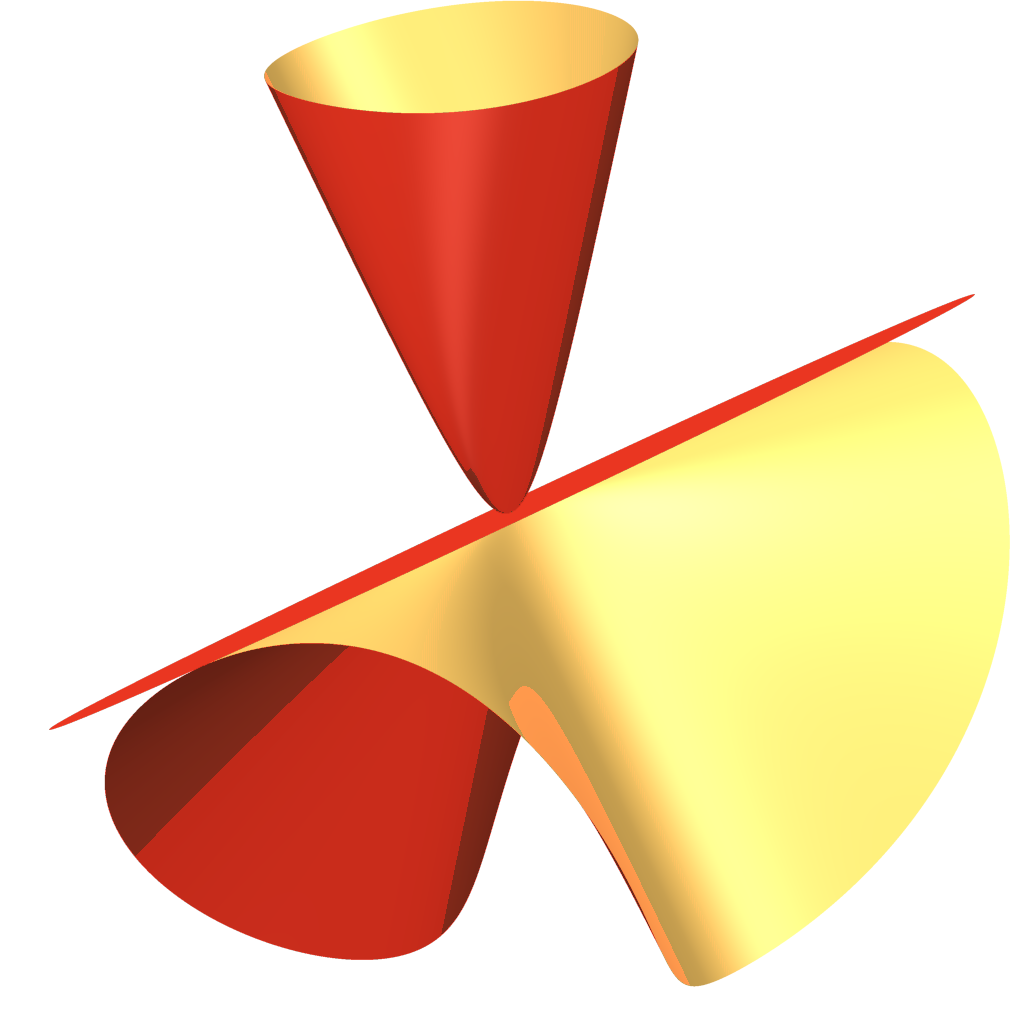}
    \caption{Original surface}
    \label{fig:dJex2}
  \end{subfigure}
  %
	%
  \begin{subfigure}[b]{0.3\textwidth}
    \includegraphics[width=\textwidth]{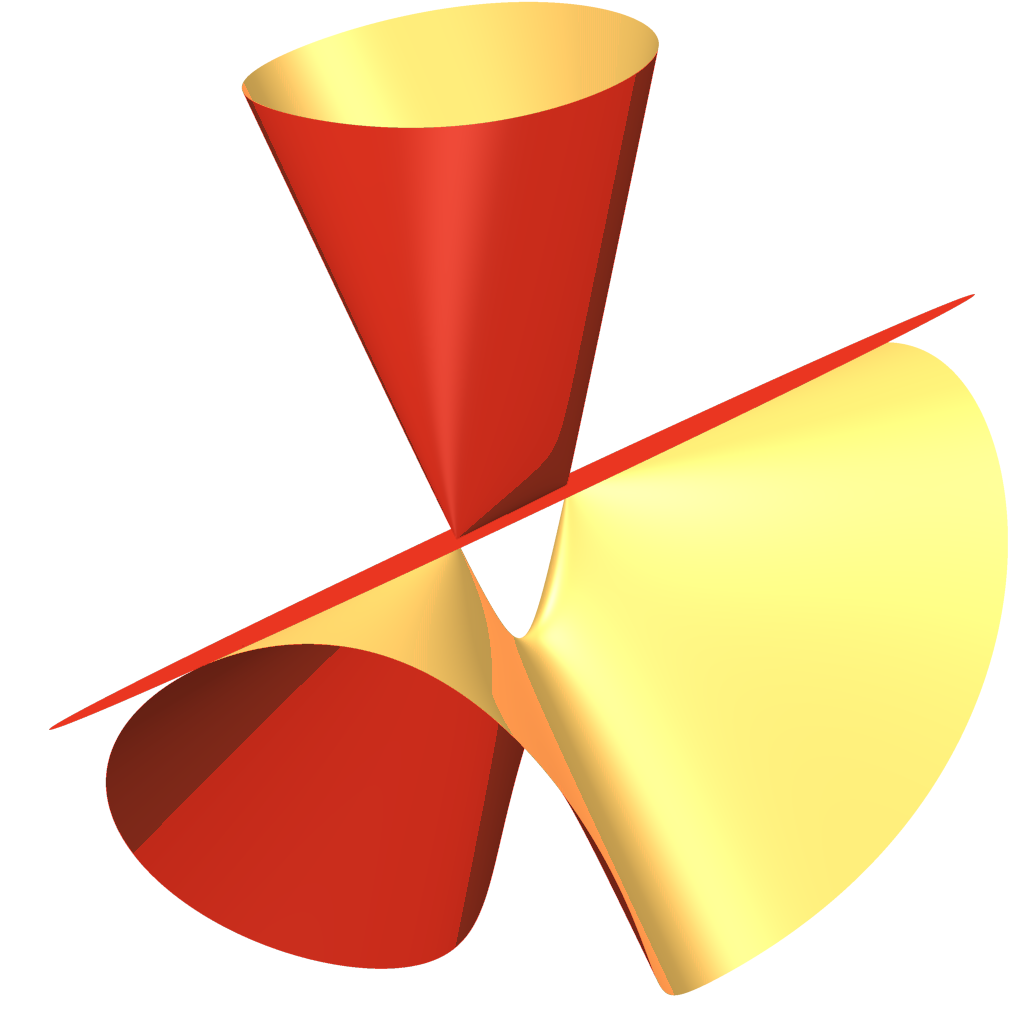}
    \caption{Deformed surface}
    \label{fig:dJex2d}	
  \end{subfigure}
\caption{Deformation $ f_s =(x^k-s) z^2 + y z^2 + y^2 z$, (\tiny{produced with Surfer software})}
\end{figure}

\subsection{More general types}
We show next that the above method is not restricted to the De Jong classes.
Consider $f = z^2x^m-z^{m+2} + z y^{m+1}$. It has the properties:
$F \simeq S^1$ ; 
$\Sigma$ is smooth; transversal type is $A_{2m+1}$;
$A = h^m$, where $h$ is the Milnor monodromy of $A_{2m+1}$.

Note that $\dim \ker (A-I) \ge 1$, and $=1$ in many cases, e.g. $m=2,3,4,5$.
This function $f$ appears as `building block' in the following deformation:\\
$g_s = z^2(x^2-s)^m - z^{m+2} + z y^{m+1}.$\\
This deformation contains two special points of the type $f$ (and no others, except isolated singularities).
If one applies the same procedure as above one gets $b_1(G) = 1$ where $G$ is the Milnor fibre of $g_0$. 
Details are left to the reader.

\begin{remark}
The fact that the first Betti number of the Milnor fibre is non-zero can also be deduced from Van Straten's \cite[Theorem 4.4.12]{vS}:
\emph{
Let $f: (\bC^3,0) \to (\bC,0)$ be a germ of a function without multiple factors, let $F$ be the Milnor fibre of $f$. Then
$$ b_1(F) \geq \# \mbox{\rm \{irreducible components of} \; f=0 \}.$$}
\end{remark}

\subsection{Deformation with triple points}
Let $f_s = x y z (x +y + z -s )$. This defines a deformation of a central arrangement with 4 hyperplanes. We get $\Sigma_i = \bP^1$ (6 copies). 
There are 4 triple points $T_{\infty,\infty,\infty}$ and one $A_1$-point. The maps $j_{1,q} : \bZ^3 \to \bZ^2$ can be described by $j_{1,q}(a,b,c) = (a+c,b+c)$. The map $j_2$ restricts to an isomorphism $\bZ \to \bZ$ on each component. 
We have all  information of the resulting map $j: \bZ^{12} \to \bZ^{14}$ up to the signs of the isomorphisms. From this we get $H_1(F;\bZ_2) =  \bZ_{2}^3$.
Compare with the dissertation \cite{Wi}, where Williams showed in particular that $H_1(F;\bZ) =  \bZ^3$.

\subsection{The class of singularities with $b_n=0$}
Most of the singularities above have $b_{n-1}= 0$ or small. What happens if $b_n=0$ ?  Examples are the product of an isolated singularity with a smooth line (such as $A_{\infty})$ and some of the functions mentioned above (e.g. $F_2 A_3$). Very few is known about this class. We can show the following ``non-splitting property'' w.r.t. isolated singularities:

\begin{proposition}
If $\hat{f}$ has the property, that $b_n(\hat{F}) = 0$, then any admissible deformation has no isolated critical points.
\end{proposition}
\begin{proof}
Note that in \ref{eq:directsumdecomp2} we have  $H_*(E, F)=0$. It follows, that $H_*(\cX\cup \cY, \cA\cup \cB)=0$ and $\oplus_{r\in R} H_*(E_r, F_r) = 0$. Therefore the set $R$ is empty.
\end{proof}


\end{document}